\documentclass[smallextended]{svjour3}

\smartqed

\usepackage{amsmath}
\usepackage{amsfonts, amssymb}
\usepackage{amscd}

\newcommand{\Z}{\mathbb{Z}}

\newcommand{\C}{{\cal C}}
\newcommand{\D}{{\cal D}}

\newcommand{\pr}{\indent{\em Proof: \ }}
\newenvironment{demo}{\noindent {\pr}\ }{\qed \medskip}

\begin{document}

\title{A characterization of $\Z_2\Z_2[u]$-linear codes
\thanks{This work has been partially
supported by the Spanish MINECO grants TIN2016-77918-P and MTM2015-69138-REDT,
and by the Catalan AGAUR grant 2014SGR-691.}}

\titlerunning{A characterization of $\Z_2\Z_2[u]$-linear codes}

\author{Joaquim Borges}

\authorrunning{J. Borges}

\institute{Joaquim Borges is with the
Dept. of Information and Communications Engineering, Universitat Aut\`{o}noma de
Barcelona, Spain.
\email{jborges@deic.uab.cat}
}

\date{Received: date / Accepted: date}

\maketitle

\begin{abstract}
We prove that the class of $\Z_2\Z_2[u]$-linear codes is exactly the class of $\Z_2$-linear codes with automorphism group of even order. Using this characterization, we give examples of known codes, e.g. perfect codes, which has a nontrivial $\Z_2\Z_2[u]$ structure. We also exhibit an example of a $\Z_2$-linear code which is not $\Z_2\Z_2[u]$-linear. Also, we state that duality of $\Z_2\Z_2[u]$-linear codes is the same that duality of $\Z_2$-linear codes.

Finally, we prove that the class of $\Z_2\Z_4$-linear codes which are also $\Z_2$-linear is strictly contained in the class of $\Z_2\Z_2[u]$-linear codes.


\end{abstract}

\section{Introduction}

$\Z_2\Z_4$-linear codes were first introduced in \cite{PuRi} as abelian translation-invariant propelinear codes. Later, in \cite{AddDual}, a comprehensive description of
$\Z_2\Z_4$-linear codes appeared. In \cite{AddDual}, the duality of such codes is studied, an appropriate inner product is defined and it is stated that the $\Z_2\Z_4$-dual code is not the same as the standard orthogonal code, that is, using the standard inner product of binary vectors. Any $\Z_2\Z_4$-linear code $C$ is a binary image of a $\Z_2\Z_4$-additive code $\C$, that is, an additive subgroup of $\Z_2^\alpha\times\Z_4^\beta$. We say that $C$ (and also $\C$) has parameters $(\alpha, \beta)$.

Recently, $\Z_2\Z_2[u]$-linear codes with parameters $(\alpha,\beta)$ have been introduced in \cite{AbSi}. They are binary images of $\Z_2\Z_2[u]$-additive codes, which are submodules of the ring $\Z_2^\alpha\times\Z_2[u]^\beta$. These codes have some similarities with $\Z_2\Z_4$-linear codes. However, there is a key difference: every $\Z_2\Z_2[u]$-linear code is also $\Z_2$-linear, which is not true, in general, for $\Z_2\Z_4$-linear codes.

The aim of this paper is to clarify the relation among all these classes. Specifically, we prove that a $\Z_2$-linear code is $\Z_2\Z_2[u]$-linear if and only if its automorphism group has even order. We also show that for a $\Z_2\Z_2[u]$-linear code, its $\Z_2\Z_2[u]$-dual code is exactly its $\Z_2$-dual code, that is, its standard binary dual code. This, in turn, implies directly that the dual weight distributions are related by MacWilliams identity. This fact was proved in \cite{AbSi}. By using these properties, we find $\Z_2\Z_2[u]$ structures for all binary linear perfect codes. In particular, for any binary linear perfect code $C$, we compute the possible values of $\alpha$ and $\beta$ such that $C$ is a $\Z_2\Z_2[u]$-linear code with parameters $(\alpha,\beta)$.

If $C$ is a $\Z_2\Z_4$-linear code with parameters $(\alpha,\beta)$, which is also $\Z_2$-linear, then we prove that $C$ has a $\Z_2\Z_2[u]$ structure with the same parameters $(\alpha,\beta)$.
In addition, we give an example showing that there are $\Z_2\Z_2[u]$-linear codes which are not $\Z_2\Z_4$-linear.

The paper is organized as follows. In the next section, we give basic
definitions and concepts. In Section \ref{Duality}, we prove that for a given $\Z_2\Z_2[u]$-linear code $C$, its $\Z_2\Z_2[u]$-dual code is exactly $C^\perp$, i.e. the standard binary orthogonal code. In Section \ref{characterization}, we study the conditions for a $\Z_2$-linear code to be $\Z_2\Z_2[u]$-linear. Moreover, we characterize $\Z_2\Z_2[u]$-linear codes as $\Z_2$-linear codes with automorphism group of even order. In Section \ref{perfect}, we prove that all $\Z_2$-linear perfect codes are $\Z_2\Z_2[u]$-linear wit parameters $(\alpha,\beta)$, where $\beta > 0$. In addition, we compute the possible values of $\alpha$ and $\beta$. In Section \ref{Z2Z4}, we analyze the relation to $\Z_2\Z_4$-linear codes. In particular, we prove that if $C$ is $\Z_2$-linear and $\Z_2\Z_4$-linear with parameters $(\alpha,\beta)$, then $C$ is also a $\Z_2\Z_2[u]$-linear code with the same parameters $(\alpha,\beta)$. We note that the reciprocal statement is not true. Finally, in Section \ref{conclusions}, we give some conclusions about the meaningful of $\Z_2\Z_2[u]$-linear codes and we point out some possible further research on the topic.

\section{Preliminaries}

Denote by ${\mathbb{Z}}_2$ and ${\mathbb{Z}}_4$ the rings of integers modulo 2
and modulo 4,
respectively. A binary code of length $n$ is any non-empty subset $C$ of
${\mathbb{Z}}_2^n$. If that
subset is a vector space then we say that it is a $\Z_2$-linear code (or binary linear code).

For any binary code $C$, an automorphism of $C$ is a coordinate permutation that leaves $C$ invariant. The automorphism group of $C$, denoted $Aut(C)$, is the group of all automorphisms of $C$.

Any
non-empty subset ${\cal C}$
of ${\mathbb{Z}}_4^n$ is a quaternary code of length $n$, and an additive
subgroup of ${\mathbb{Z}}_4^n$ is
called a quaternary linear code. The elements of a code are called
codewords.


The classical Gray map
$\phi:\;\Z_4\;\longrightarrow\;\Z_2^{2}$ is defined by
$$
\phi(0)=(0,0),\;\;\phi(1)=(0,1),\;\;\phi(2)=(1,1),\;\;\phi(3)=(1,0).
$$
If $a=(a_1,\ldots,a_m)\in\Z_4^m$, then the Gray map of $a$ is the coordinate-wise
extended map $\phi(a)=(\phi(a_1),\ldots,\phi(a_m))$. We naturally extend the
Gray map for vectors ${\bf u}=(u\mid  u')\in
{\mathbb{Z}}_2^\alpha\times{\mathbb{Z}}_4^\beta$ so that $\Phi({\bf u}) = (u\mid
\phi(u'))$.


\begin{definition}
A $\mathbb{Z}_2\mathbb{Z}_4$-additive code $\C$ with parameters $(\alpha,\beta)$ is an additive subgroup of
$\mathbb{Z}_2^\alpha\times\mathbb{Z}_4^\beta$.
\end{definition}

Such codes are extensively
studied in \cite{AddDual}.
Alternatively, we can define a $\mathbb{Z}_2\mathbb{Z}_4$-additive code as a $\Z_4$-submodule of $\mathbb{Z}_2^\alpha\times\mathbb{Z}_4^\beta$, where the scalar product $\lambda {\bf x}$ is defined as ${\bf x}+\cdots +{\bf x}$, $\lambda$ times (of course, if $\lambda=0$, then $\lambda {\bf x}=0$), for $\lambda\in\Z_4$, ${\bf x}\in \mathbb{Z}_2^\alpha\times\mathbb{Z}_4^\beta$.

If $\C$ is a $\mathbb{Z}_2\mathbb{Z}_4$-additive code with parameters $(\alpha,\beta)$, then the binary image $C=\Phi(\C)$ is called a $\mathbb{Z}_2\mathbb{Z}_4$-linear code with parameters $(\alpha,\beta)$. Note that $C$ is a binary code of length $n=\alpha+2\beta$, but $C$ is not $\Z_2$-linear, in general \cite{AddDual}. If $\alpha=0$, then $C$ is called a $\Z_4$-linear code. If $\beta=0$, then $C$ is simply a $\Z_2$-linear code.

%


The standard inner product in $\mathbb{Z}_2^\alpha\times\mathbb{Z}_4^\beta$,
defined in \cite{AddDual}, can be written as
$$\textbf{u}\cdot \textbf{v} =
2\left(\sum_{i=1}^{\alpha}u_iv_i\right)+\sum_{j=1}^{\beta}u'_jv'_j\in
\mathbb{Z}_4,$$
where the computations are made taking the zeros and ones in the $\alpha$ binary
coordinates as quaternary zeros and ones, respectively. The $\Z_2\Z_4$-dual code of
a $\mathbb{Z}_2\mathbb{Z}_4$-additive code ${\cal C}$ is defined in the
standard way by
$${\cal C}^\perp=\{\textbf{v} \in \mathbb{Z}_2^\alpha\times\mathbb{Z}_4^\beta
\mid \textbf{u}\cdot\textbf{v}=0, \mbox{ for all }\textbf{u}\in{\cal C}\}.$$

%

\bigskip

Consider the ring $\Z_2[u]=\Z_2 + u\Z_2=\{0,1,u,1+u\}$, where $u^2=0$. Note that $(\Z_2[u],+)$ is group-isomorphic to the Klein group $(\Z_2^2,+)$. But with the product operation, $(\Z_2[u],\cdot )$ is monoid-isomorphic to $(\Z_4,\cdot)$. Define the map $\pi:\;\Z_2[u]\;\longrightarrow\;\Z_2$, such that $\pi(0)=\pi(u)=0$ and $\pi(1)=\pi(1+u)=1$. Then, for $\lambda\in\Z_2[u]$ and ${\bf x}=(x_1,\ldots,x_\alpha\mid x'_1,\ldots,x'_\beta)\in \Z_2^\alpha\times\Z_2[u]^\beta$, we can consider the scalar product
$$
\lambda {\bf x}= (\pi(\lambda)x_1,\ldots,\pi(\lambda)x_\alpha\mid \lambda x'_1,\ldots,\lambda x'_\beta)\in \Z_2^\alpha\times\Z_2[u]^\beta.
$$
With this operation, $\Z_2^\alpha\times\Z_2[u]^\beta$ is a $\Z_2[u]$-module. Note that, a $\Z_2[u]$-submodule of $\Z_2^\alpha\times\Z_2[u]^\beta$ is not the same as a subgroup of $\Z_2^\alpha\times\Z_2[u]^\beta$.

\begin{definition}[\cite{AbSi}]
A $\Z_2\Z_2[u]$-additive code $\C$ with parameters $(\alpha,\beta)$ is a $\Z_2[u]$-submodule of $\mathbb{Z}_2^\alpha\times\mathbb{Z}_2[u]^\beta$.
\end{definition}

The following straightforward equivalence can be used as an alternative definition.

\begin{lemma}\label{equiv}
A code $\C\subseteq \mathbb{Z}_2^\alpha\times\mathbb{Z}_2[u]^\beta$ is $\Z_2\Z_2[u]$-additive if and only if
\begin{eqnarray*}\label{definicio}
&&u {\bf z}\in\C\;\;\;\forall {\bf z}\in\C,\mbox{ and }\\
&&{\bf x}+{\bf y}\in\C\;\;\;\forall {\bf x}, {\bf y}\in \C.
\end{eqnarray*}
\end{lemma}

As for $\mathbb{Z}_2^\alpha\times\mathbb{Z}_4^\beta$, we can also define a Gray-like map. Let
$\psi:\;\Z_2[u]\;\longrightarrow\;\Z_2^{2}$ be defined as
$$
\psi(0)=(0,0),\;\;\psi(1)=(0,1),\;\;\psi(u)=(1,1),\;\;\psi(1+u)=(1,0).
$$
If $a=(a_1,\ldots,a_m)\in\Z_2[u]^m$, then the  coordinate-wise extension of $\psi$ is
$\psi(a)=(\psi(a_1),\ldots,\psi(a_m))$. Now, we define the
Gray-like map for elements ${\bf u}=(u\mid  u')\in
{\mathbb{Z}}_2^\alpha\times{\mathbb{Z}}_2[u]^\beta$ so that $\Psi({\bf u}) = (u\mid
\psi(u'))$.

If $\C$ is a $\mathbb{Z}_2\mathbb{Z}_2[u]$-additive code with parameters $(\alpha,\beta)$, then the binary image $C=\Psi(\C)$ is called a $\mathbb{Z}_2\mathbb{Z}_2[u]$-linear code with parameters $(\alpha,\beta)$. Note that, unlike for $\Z_2\Z_4$-linear codes, $C$ is a $\Z_2$-linear code of length $n=\alpha+2\beta$. This fact is clear since for any pair of elements ${\bf x},{\bf y}\in \mathbb{Z}_2^\alpha\times\mathbb{Z}_2[u]^\beta$, we have that
$\Psi({\bf x}) +\Psi({\bf y})=\Psi({\bf x} + {\bf y})$.

The inner product in $\mathbb{Z}_2^\alpha\times\mathbb{Z}_2[u]^\beta$,
defined in \cite{AbSi}, can be written as
$$\textbf{u}\cdot \textbf{v} =
u\left(\sum_{i=1}^{\alpha}u_iv_i\right)+\sum_{j=1}^{\beta}u'_jv'_j\in
\mathbb{Z}_2[u],$$
where the computations are made taking the zeros and ones in the $\alpha$ binary
coordinates as zeros and ones in $\Z_2[u]$, respectively. The $\mathbb{Z}_2\mathbb\Z_2[u]$-dual code of
a $\mathbb{Z}_2\mathbb\Z_2[u]$-additive code ${\cal C}$ is defined in the
standard way by
$${\cal C}^\perp=\{\textbf{v} \in \mathbb{Z}_2^\alpha\times\mathbb\Z_2[u]^\beta
\mid \textbf{u}\cdot\textbf{v}=0, \mbox{ for all }\textbf{u}\in{\cal C}\}.$$

\section{Duality of $\Z_2\Z_2[u]$-linear codes}\label{Duality}

It is readily verified that if $a,b\in\Z_2[u]$, then $\psi(a)\cdot \psi(b) =1$ if and only if $ab\in\{1,u\}$. This property can be easily generalized for elements in $\Z_2[u]^\beta$.

\begin{lemma}\label{tecnic}
If $x',y'\in\Z_2[u]^\beta$, then $\psi(x')\cdot \psi(y') =1$ if and only if $x'\cdot y'\in\{1,u\}$.
\end{lemma}

\begin{proof}
 Each pair of equal addends in $x'\cdot y'$ gives $0$. Thus, we can omit all these pairs. The set of nonzero remaining terms is:
 \begin{itemize}
 \item[(i)] $\{1,u,1+u\}$ or $\emptyset$, if $x'\cdot y'=0$.
 \item[(ii)] $\{1\}$ or $\{u, 1+u\}$, if $x'\cdot y'=1$.
 \item[(iii)] $\{u\}$ or $\{1, 1+u\}$, if $x'\cdot y'=u$.
 \item[(iv)] $\{1+u\}$ or $\{1, u\}$, if $x'\cdot y'=1+u$.
 \end{itemize}
 Clearly, cases (i) and (iv) give $\psi(x')\cdot \psi(y')=0$, whereas cases (ii) and (iii) give $\psi(x')\cdot \psi(y')=1$.
\end{proof}

\begin{proposition}\label{producte}
Let ${\bf x},{\bf y}\in \Z_2^\alpha\times\Z_2[u]^\beta$.
\begin{itemize}
\item[(i)] If ${\bf x}\cdot {\bf y}=0$, then $\Psi ({\bf x})\cdot \Psi ({\bf y})=0$.
\item[(ii)] If ${\bf x}\cdot {\bf y}\neq 0$ and $\Psi ({\bf x})\cdot \Psi ({\bf y})=0$, then $\Psi ({\bf x})\cdot \Psi ((1+u){\bf y})=1$.
\end{itemize}
\end{proposition}

\begin{demo}
Let ${\bf x}=(x\mid x')$ and ${\bf y}=(y\mid y')$. We can write the inner product of ${\bf x}$ and ${\bf y}$ as ${\bf x}\cdot {\bf y} = u(x\cdot y) + (x'\cdot y')$.

\begin{itemize}
\item[(i)] If ${\bf x}\cdot {\bf y}=0$, then either (a) $x\cdot y=x' \cdot y'=0$, or (b) $x\cdot y =1$ and $x'\cdot y'=u$.
\begin{itemize}
\item[(a)] By Lemma \ref{tecnic}, we have that $\Psi(x')\cdot \Psi(y')=0$ and hence $\Psi ({\bf x})\cdot \Psi ({\bf y})=0$.
\item[(b)] Again, By Lemma \ref{tecnic}, we obtain $\Psi(x')\cdot \Psi(y')=1$ and then $\Psi ({\bf x})\cdot \Psi ({\bf y})=0$.
\end{itemize}

\item[(ii)] If ${\bf x}\cdot {\bf y}\neq 0$, then either (a) $x\cdot y=0$ and $x'\cdot y'\neq 0$, or (b) $x\cdot y=1$ and $x'\cdot y'\neq u$.
\begin{itemize}
\item[(a)] In this case $x'\cdot y'\in\{1,u,1+u\}$. But $x\cdot y=0$ and $\Psi ({\bf x})\cdot \Psi ({\bf y})=0$ imply that $\Psi(x')\cdot\Psi(y')=0$ and hence, by Lemma \ref{tecnic}, the only possible case is that $x'\cdot y'= 1+u$. Therefore, $x'\cdot ((1+u)y')=1$ and $\Psi(x')\cdot\Psi((1+u)y')=1$, again by Lemma \ref{tecnic}. Thus, $\Psi ({\bf x})\cdot \Psi ((1+u){\bf y})=1$.
\item[(b)] We have $x'\cdot y'\in\{0,1,1+u\}$.  Since $x\cdot y=1$ and $\Psi ({\bf x})\cdot \Psi ({\bf y})=0$, we obtain $\Psi(x')\cdot\Psi(y')=1$. By Lemma \ref{tecnic}, the only possibility is $x'\cdot y'=1$. Hence, $x'\cdot ((1+u)y')=1+u$ and $\Psi(x')\cdot\Psi((1+u)y')=0$. We conclude $\Psi ({\bf x})\cdot \Psi ((1+u){\bf y})=1$.
\end{itemize}
\end{itemize}
\end{demo}

\begin{corollary}\label{dualigual}
Let $\C$ be a $\Z_2\Z_2[u]$-additive code and let $C=\Psi(\C)$ be the corresponding binary $\Z_2\Z_2[u]$-linear code. Then, $\Psi(\C^\perp)=C^\perp$.
\end{corollary}

\begin{demo}
If ${\bf x}\in \C^\perp$, then ${\bf x}\cdot {\bf c}=0$, for all ${\bf c}\in \C$. Hence, by Proposition \ref{producte}(i), we have that $\Psi({\bf x})\cdot \Psi ({\bf c})=0$, for all ${\bf c}\in \C$, implying that $\Psi({\bf x})\in C^\perp$. We have proved $\Psi(\C^\perp)\subseteq C^\perp$.

If ${\bf x}\notin \C^\perp$, then ${\bf x}\cdot {\bf c}\neq 0$, for some ${\bf c}\in \C$. Now, by Proposition \ref{producte}(ii), we have that $\Psi({\bf x})\cdot \Psi ({\bf c})\neq 0$ or $\Psi({\bf x})\cdot \Psi ((1+u){\bf c})\neq0$. It follows that $\Psi({\bf x})\notin C^\perp$ and therefore $C^\perp \subseteq \Psi(\C^\perp)$.
\end{demo}

Obviously, this immediately implies that the weight distributions of $\C$ and $\C^\perp$ are related by MacWilliams identity, as it was proved in \cite{AbSi}.

To finish this section, we prove that the dual of a $\Z_2\Z_2[u]$-linear code is also $\Z_2\Z_2[u]$-linear with the same parameters.

\begin{proposition}\label{dualZ2Zu}
A binary code $C$ is $\Z_2\Z_2[u]$-linear with parameters $(\alpha,\beta)$ if and only if $C^\perp$ is $\Z_2\Z_2[u]$-linear with the same parameters $(\alpha,\beta)$.
\end{proposition}

\begin{demo}
Since $(C^\perp)^\perp=C$, it is enough to prove the `only if' part. Assume that $C$ is a $\Z_2\Z_2[u]$-linear code with parameters $(\alpha,\beta)$. Let $\C^\perp=\Psi^{-1}(C^\perp)$.
By linearity of $C^\perp$ and Lemma \ref{equiv}, we only need to proof that $u \Psi^{-1}(c)\in \C^\perp$, for all $c\in C^\perp$. For any codeword ${\bf x}\in \C$, we have $\left(u\Psi^{-1}(c)\right)\cdot {\bf x}=u\left(\Psi^{-1}(c)\cdot {\bf x}\right)=u0=0$, which implies $u\Psi^{-1}(c)\in \C^\perp$.
\end{demo}

\section{Characterization of $\Z_2\Z_2[u]$-linear codes}\label{characterization}

Given a $\Z_2$-linear code $C$ of length $n$, a natural question is if we can choose a set of $\beta$ pairs of coordinates such that $C$ is a $\Z_2\Z_2[u]$-linear code with parameters $(n-2\beta,\beta)$. The next lemma shows us that it is enough to answer the question for a generator matrix of $C$.

\begin{lemma}\label{generadors}
Let $S\subset {\mathbb{Z}}_2^\alpha\times{\mathbb{Z}}_2[u]^\beta$ and let $C$ be the $\Z_2$-linear code generated by the binary image vectors of $S$, $C=<\Psi(S)>$. Then, $C$ is a $\Z_2\Z_2[u]$-linear code with parameters $(\alpha,\beta)$ if and only if $\Psi(u {\bf x})\in C$, for all ${\bf x}\in S$.
\end{lemma}

\begin{demo}
Let $\C=\Psi^{-1}(C)$. Then, $C$ is $\Z_2\Z_2[u]$-linear if and only if $\C$ is $\Z_2\Z_2[u]$-additive. Applying Lemma \ref{equiv}, we have that $\C$ is a $\Z_2\Z_2[u]$-additive code if and only $u{\bf z}\in\C$, for all ${\bf z}\in \C$. Let $w=\Psi({\bf z})$, thus $w=\Psi({\bf x_1})+ \cdots +\Psi({\bf x_r})$, for some ${\bf x_1}, \ldots , {\bf x_r}\in S$. Hence ${\bf z}={\bf x_1}+\cdots +{\bf x_r}$ and $u{\bf z}=u{\bf x_1}+\cdots +u{\bf x_r}$.
Therefore, $\C$ is $\Z_2\Z_2[u]$-additive if and only if $u{\bf x_i} \in \C$, for all $i=1,\ldots,r$.
\end{demo}

Now, we give a necessary and sufficient condition for a $\Z_2$-linear code to be $\Z_2\Z_2[u]$-linear.

\begin{proposition}\label{involucio}
Let $C$ be a $\Z_2$-linear code. Then, $C$ is a $\Z_2\Z_2[u]$-linear code with parameters $(\alpha, \beta)$, where $\beta > 0$, if and only if there exists an involution $\sigma\in Aut(C)$ fixing $\alpha$ coordinates.
\end{proposition}

\begin{demo}
Assume that $C$ is a $\Z_2\Z_2[u]$-linear code with $\beta > 0$. Let $\C=\Psi^{-1}(C)$ and for any codeword ${\bf x}=(x_1,\ldots,x_\alpha \mid x'_1,\ldots,x'_\beta)\in \C$, we write its binary image as
$x=(x_1,\ldots,x_\alpha\mid y_1,\ldots,y_{2\beta})$, where $\psi(x'_i)=(y_{2i-1},y_{2i})$, for $i=1,\ldots,\beta$. Let $\sigma$ be the involution that transposes $y_{2i-1}$ and $y_{2i}$, for all $i=1,\ldots,\beta$. Clearly, $\Psi\left((1+u){\bf x}\right)=\sigma(x)$. Since $(1+u){\bf x}\in \C$, we have that $\sigma\in Aut(C)$.

Reciprocally, if $\sigma\in Aut(C)$ has order $2$, then $\sigma$ is a product of disjoin transpositions. Considering the pairs of coordinates that $\sigma$ transposes as the images of $\Z_2[u]$ coordinates, we obtain that $\sigma(x)=\Psi\left((1+u)\Psi^{-1}(x)\right)$, for any codeword $x\in C$. Since $\sigma(x)\in C$, we have that $(1+u){\bf x}\in \C=\Psi^{-1}(C)$, for any ${\bf x}\in \C$. But this condition implies that $\C$ is a $\Z_2\Z_2[u]$-additive code since $(1+u){\bf x}={\bf x} + u{\bf x}$ and thus $u{\bf x}\in\C$. The result follows applying Lemma \ref{equiv}.
\end{demo}

\begin{corollary}\label{autgrup}
A $\Z_2$-linear code $C$ is $\Z_2\Z_2[u]$-linear with parameters $(\alpha, \beta)$, where $\beta > 0$, if and only if $Aut(C)$ has even order.
\end{corollary}

\begin{demo}
$Aut(C)$ has even order if and only if $Aut(C)$ contains an involution. The statement then follows by Proposition \ref{involucio}.
\end{demo}

\begin{remark}
Note that Proposition \ref{involucio} and Corollary \ref{autgrup} applies also to $\Z_2\Z_4$-linear codes but only in one direction. I.e., if $C$ is a $\Z_2\Z_4$-linear code, then $Aut(C)$ has even order. But the converse is not true, in general.
\end{remark}

\begin{example}
Consider the code $C$ with generator matrix
$$
\left(
  \begin{array}{cccccccccccc}
    1 & 1 & 1 & 0 & 0 & 1 & 0 & 0 & 0 & 0 & 0 & 0 \\
    0 & 0 & 1 & 1 & 1 & 0 & 0 & 0 & 0 & 0 & 0 & 0 \\
    0 & 0 & 0 & 0 & 1 & 1 & 1 & 0 & 0 & 0 & 0 & 0 \\
    0 & 0 & 0 & 0 & 0 & 0 & 1 & 1 & 1 & 0 & 0 & 0 \\
    0 & 0 & 0 & 0 & 0 & 0 & 0 & 0 & 1 & 1 & 1 & 0 \\
    1 & 0 & 0 & 0 & 0 & 0 & 0 & 0 & 0 & 0 & 1 & 1 \\
  \end{array}
\right).
$$
As it is pointed out in \cite[Problem (32), p. 230]{MacW}, $Aut(C)$ is trivial, i.e. it only contains the identity permutation. Therefore, $C$ is not $\Z_2\Z_2[u]$-linear for $\beta > 0$.
\end{example}

In the next section we see several examples of well-known codes with a $\Z_2\Z_2[u]$ structure.

\section{$\Z_2\Z_2[u]$-linear perfect codes}\label{perfect}

A binary {\em repetition} code $C=\{(0,\ldots, 0), (1,\ldots, 1)\}$ of odd length $n$ is a trivial perfect code. Its dual code is the {\em even} code which contains all vectors of length $n$ and even weight (i.e. with an even number of nonzero coordinates). Clearly, these codes can be considered as $\Z_2\Z_2[u]$-linear codes with parameters $(n-2\beta,\beta)$, for all $\beta\in\{0,\ldots, (n-1)/2\}$.

It is well known that the binary linear perfect codes with more than two codewords are:

(1) The binary {\em Hamming} 1-perfect codes of length $n=2^t-1$ ($t\ge 3$), dimension $k=2^t-t-1$ and minimum distance $d=3$.

(2) The binary {\em Golay} 3-perfect code of length $n=23$, dimension $k=12$ and minimum distance $d=7$.

In this section we prove that these codes are $\Z_2\Z_2[u]$-linear codes.

\bigskip

Let $H_t$ be a Hamming code of length $n=2^t-1$, where $t\geq 3$. The dual code $H_t^\perp$ is known as the {\em simplex} code. It is a constant-weight code with all nonzero codewords of weight $2^{t-1}$. A parity-check matrix $M_r$ for $H_t$ (which is a generator matrix for $H_t^\perp$) contains all nonzero column vectors of length $t$.

%
%

\begin{theorem}\label{hamming}
Let $H_t$ be a Hamming code of length $n=2^t-1$. Then,
$H_t$ is a $\Z_2\Z_2[u]$-linear code with parameters $(2^r-1,2^{t-1}-2^{r-1})$, for all $r$ such that $t/2 \leq r\leq t$.
\end{theorem}

\begin{proof}
The case $r=t$ corresponds to the trivial case $(\alpha,\beta)$=$(n,0)$.
In \cite{CPV}, it is shown that $Aut(H_t)$ contains involutions fixing $2^r-1$ points for $t/2 \leq r\leq t$. Thus, the statement follows by Proposition \ref{involucio}. 
\end{proof}

\begin{example}
A parity-check matrix for $H_3$ is
$$
M_3=\left(
  \begin{array}{ccccccc}
    0 & 0 & 0 & 1 & 1 & 1 & 1 \\
    0 & 1 & 1 & 0 & 0 & 1 & 1 \\
    1 & 0 & 1 & 0 & 1 & 0 & 1 \\
  \end{array}
\right).
$$

We can take the pairs of coordinates $(4,5)$ and $(6,7)$ as $\Z_2[u]$ coordinates and consider the $\Z_2\Z_2[u]$-additive code $\C$ generated by
$$
\left(
  \begin{array}{ccccc}
    0 & 0 & 0 & u & u \\
    0 & 1 & 1 & 0 & u \\
    1 & 0 & 1 & 1 & 1 \\
  \end{array}
\right).
$$
Note that multiplying $u$ by any row gives the allzero vector or a weight $4$ vector whose binary image is in $H_3^\perp$.
Thus, $\Psi(\C)=H_3^\perp$. Hence, by Corollary \ref{dualigual} and Proposition \ref{dualZ2Zu}, $H_3$ is a $\Z_2\Z_2[u]$-linear code with parameters $(3,2)$. We remark that $H_3$ is also a $\Z_2\Z_4$-linear code with the same parameters. But, according to \cite{Add}, $H_t$ is not $\Z_2\Z_4$-linear for $\beta>0$ and $t>3$.
\end{example}

\begin{example}
Consider the parity-check matrix for $H_4$
$$
M_4=\left(
  \begin{array}{ccccccccccccccc}
    0 & 0 & 0 & 0 & 0 & 0 & 0 & 1 & 1 & 1 & 1 & 1 & 1 & 1 & 1 \\
    0 & 0 & 0 & 1 & 1 & 1 & 1 & 0 & 0 & 1 & 1 & 0 & 0 & 1 & 1 \\
    0 & 1 & 1 & 0 & 0 & 1 & 1 & 0 & 1 & 0 & 1 & 0 & 1 & 0 & 1 \\
    1 & 0 & 1 & 0 & 1 & 0 & 1 & 0 & 0 & 0 & 0 & 1 & 1 & 1 & 1 \\
  \end{array}
\right).
$$
Again, we can take the pairs of coordinates $(8,9)$, $(10,11)$, $(12,13)$ and $(14,15)$ as $\Z_2[u]$ coordinates. Let $\C$ be the $\Z_2\Z_2[u]$-additive code generated by
$$
\left(
  \begin{array}{ccccccccccc}
    0 & 0 & 0 & 0 & 0 & 0 & 0 & u & u & u & u \\
    0 & 0 & 0 & 1 & 1 & 1 & 1 & 0 & u & 0 & u \\
    0 & 1 & 1 & 0 & 0 & 1 & 1 & 1 & 1 & 1 & 1 \\
    1 & 0 & 1 & 0 & 1 & 0 & 1 & 0 & 0 & u & u \\
  \end{array}
\right).
$$

Multiplying any row by $u$ gives the allzero vector or a weight $8$ vector whose binary image is in $H_4^\perp$. Therefore, $\Psi(\C)=H_4^\perp$ and $H_4$ is a $\Z_2\Z_2[u]$-linear code with parameters $(7,4)$. Note that, taking the same pairs of coordinates as quaternary coordinates, it is also true that $H_4^\perp$ is a $\Z_2\Z_4$-linear code, but the $\Z_2\Z_4$-dual code is not a Hamming code. For example, the vector $v=(0,0,0,1,0,0,0,0,0,0,0,0,1,0,1)$ is not orthogonal to the third row of $M_4$. However, $v$ is in the $\Z_2\Z_4$-dual of $H^\perp_4$.

After a permutation of columns, the matrix $M_4$ can be written as
$$
\left(
  \begin{array}{ccccccccccccccc}
    1 & 1 & 0 & 0 & 1 & 0 & 1 & 1 & 1 & 1 & 0 & 0 & 1 & 0 & 0 \\
    0 & 1 & 1 & 0 & 0 & 0 & 1 & 0 & 1 & 1 & 1 & 1 & 0 & 0 & 1 \\
    1 & 1 & 0 & 1 & 0 & 1 & 0 & 1 & 1 & 0 & 1 & 1 & 0 & 0 & 0 \\
    1 & 0 & 1 & 0 & 0 & 1 & 0 & 0 & 1 & 1 & 1 & 0 & 1 & 1 & 0 \\
  \end{array}
\right).
$$

Now, taking the pairs of coordinates $(i,i+1)$ for $i=4,\ldots, 14$ as $\Z_2[u]$ coordinates, we also have that $H_4^\perp$ is the binary image of the code generated by
$$
\left(
  \begin{array}{ccccccccc}
    1 & 1 & 0 & 1 & 1 & u & 1+u & 1 & 0 \\
    0 & 1 & 1 & 0 & 1 & 1 & u & 1+u & 1 \\
    1 & 1 & 0 & 1+u & 1+u & u & 1 & 1+u & 0 \\
    1 & 0 & 1 & 0 & 1+u & 1 & u & 1 & 1+u \\
  \end{array}
\right).
$$
Therefore, $H_4$ is also a $\Z_2\Z_2[u]$-linear code with parameters $(3,6)$.
\end{example}

\begin{corollary}
The extended Hamming code $H'_t$, the dual of a Hamming code $H_t^\perp$ (simplex code), and the dual of an extended Hamming code $(H'_t)^\perp$ (linear Hadamard code) are $\Z_2\Z_2[u]$-linear codes with parameters $(2^r,2^{t-1}-2^{r-1})$, $(2^r-1,2^{t-1}-2^{r-1})$, and $(2^r,2^{t-1}-2^{r-1})$, respectively.
\end{corollary}

\begin{demo}
On the one hand, extending a $\Z_2\Z_2[u]$-linear code with parameters $(\alpha, \beta)$ trivially results in a $\Z_2\Z_2[u]$-linear code with parameters $(\alpha+1, \beta)$. On the other hand, by Proposition \ref{dualZ2Zu},
the dual code has the same parameters.
\end{demo}

\begin{theorem}
The binary Golay code $G_{23}$ and the extended binary Golay code $G_{24}$ are $\Z_2\Z_2[u]$-linear codes with parameters $(\alpha,\beta)$. For $\beta > 0$, the parameters are:
\begin{itemize}
\item[(i)] $(0,12)$ or $(8,8)$, for $G_{24}$.
\item[(ii)] $(7,8)$, for $G_{23}$.
\end{itemize}
\end{theorem}

\begin{demo}
It is well known that the automorphism groups of $G_{23}$ and $G_{24}$ are the Mathieu groups $M_{23}$ and $M_{24}$, respectively \cite{MacW}. In \cite{Tesi}, it is stated that $M_{24}$ has 43470 fixed-point-free involutions. The remaining involutions of $M_{24}$ are 11385 involutions fixing 8 points. Therefore, by Proposition \ref{involucio}, $G_{24}$ has 43470 $\Z_2\Z_2[u]$ different structures with parameters $(0,12)$ and 11385 with parameters $(8,8)$. For the case of $M_{23}$, it has 3795 involutions, all of them fixing 7 points. Therefore, $G_{23}$ has 3795 $\Z_2\Z_2[u]$ structures with parameters $(7,8)$.
\end{demo}

\section{$\Z_2\Z_2[u]$-linear and $\Z_2\Z_4$-linear codes}\label{Z2Z4}

In this section we prove that any $\Z_2\Z_4$-linear code with parameters $(\alpha,\beta)$ which is also $\Z_2$-linear has a $\Z_2\Z_2[u]$ structure with the same parameters. It is not difficult to see this property using Corollary \ref{autgrup}, however, we give here an independent proof in order to better clarify the relation between both classes of codes.

The following property was stated in \cite{Sole} for vectors over $\Z_4$. Its generalization for vectors over $\Z_2\times\Z_4$ is easy and established in \cite{RankKernel}.

\begin{lemma}\label{sumes}
Let ${\bf x}, {\bf y}\in \Z_2^\alpha\times\Z_4^\beta$. The following identity holds:
$$
\Phi({\bf x}) + \Phi({\bf y})=\Phi({\bf x}+{\bf y}) + \Phi\left(2({\bf x}\star {\bf y})\right),
$$
where $\star$ stands for the coordinate-wise product.
\end{lemma}

The next lemma \cite{RankKernel} is a direct consequence.

\begin{lemma}\label{lineal}
If $\C$ is a $\mathbb{Z}_2\mathbb{Z}_4$-additive code, then its binary image $C=\Phi(\C)$ is $\Z_2$-linear if and only if $2({\bf x}\star {\bf y})\in \C$, for all ${\bf x}, {\bf y}\in \C$.
\end{lemma}

Define the map $\theta:\;\Z_2^\alpha\times\Z_4^\beta\;\longrightarrow\;\Z_2^\alpha\times\Z_2[u]^\beta$ such that, for every element $(x_1,\ldots,x_\alpha\mid y_1,\ldots,y_\beta)\in \Z_2^\alpha\times\Z_4^\beta$,
$$
\theta(x_1,\ldots,x_\alpha\mid y_1,\ldots,y_\beta)=(x_1,\ldots,x_\alpha\mid \vartheta(y_1),\ldots,\vartheta(y_\beta)),
$$
where $\vartheta(0)=0;\;\;\vartheta(1)=1;\;\;\vartheta(2)=u;\;\;\vartheta(3)=1+u$. Note that $\theta=\Psi^{-1}\Phi$.

\begin{theorem}\label{Z2Z4Lineals}
If $\C\subseteq \Z_2^\alpha\times\Z_4^\beta$ is a $\Z_2\Z_4$-additive code such that $\Phi(\C)$ is $\Z_2$-linear, then $\C'=\theta(\C)\subseteq \Z_2^\alpha\times\Z_2[u]^\beta$ is a $\Z_2\Z_2[u]$-additive code.
\end{theorem}

\begin{demo}
We use the characterization of Lemma \ref{equiv} to prove the statement.

Given ${\bf x}\in \C'$, we need to prove that $u {\bf x}\in \C'$. Note that $u{\bf x}=\theta(2\theta^{-1}({\bf x}))$ which is in $\C'$.

Next, we want to prove that ${\bf x}+{\bf y}\in\C'$, for all ${\bf x}, {\bf y}\in \C'$. Clearly,
\begin{equation}\label{primera}
{\bf x}+{\bf y} = \Psi^{-1}\left(\Psi({\bf x})+\Psi({\bf y})\right).
\end{equation}

By Lemma \ref{sumes}, we have
\begin{equation}\label{segona}
\Psi({\bf x})+\Psi({\bf y})=\Phi\left(\Phi^{-1}\left(\Psi({\bf x})\right) + \Phi^{-1}\left(\Psi({\bf y})\right) + 2\left(\Phi^{-1}\left(\Psi({\bf x})\right)\star \Phi^{-1}\left(\Psi({\bf y})\right)\right)\right).
\end{equation}

Combining Equations \ref{primera} and \ref{segona}, we obtain
$$
{\bf x}+{\bf y}=\theta\left(\theta^{-1}({\bf x})+\theta^{-1}({\bf y})+2(\theta^{-1}({\bf x})\star\theta^{-1}({\bf y}))\right).
$$
Since $\Phi(\C)$ is $\Z_2$-linear, we have that $2\left(\theta^{-1}({\bf x})\star\theta^{-1}({\bf y})\right)\in \C$, by Lemma \ref{lineal}. It follows that ${\bf x}+{\bf y}\in\C'$.
\end{demo}

However, there are $\Z_2\Z_2[u]$-linear codes which are not $\Z_2\Z_4$-linear, as we can see in the following example.

\begin{example}
Let $\D\subset \Z_2[u]^4$ be the code generated by ${\bf x}=(1,1,1,u)$ and ${\bf y}=(1,u,1,1)$. We can see that
$$
\theta\left(\theta^{-1}({\bf x}) + \theta^{-1}({\bf y})\right)=\theta(2,3,2,3)=(u,1+u,u,1+u).
$$
It is easy to check that the equation $\lambda {\bf x}+\mu {\bf y}= (u,1+u,u,1+u)$ has no solution for $\lambda,\mu\in\Z_2[u]$. Therefore $\C=\theta^{-1}(\D)$ is not a $\Z_2\Z_4$-additive code.
\end{example}

It is worth noting that if $\C$ is a $\Z_2\Z_4$-additive code such that $\Phi(\C)$ is $\Z_2$-linear, it is not yet true that $\Phi(\C^\perp)=\Phi(\C)^\perp$ as we can see in the next example.

\begin{example}
Let $\C\subset \Z_4^3$ be the code generated by ${\bf x}=(1,1,1)$ and ${\bf y}=(0,2,3)$. It can be easily verified that $\Phi(\C)$ is $\Z_2$-linear. However, we have that $(1,1,2)\in\C^\perp$, but $\Phi(1,1,2)=(0,1,0,1,1,1)\notin \Phi(\C)^\perp$.
\end{example}

\section{Conclusions}\label{conclusions}

From Corollary \ref{autgrup}, it seems that $\Z_2\Z_4$-linear codes form a wide class of $\Z_2$-linear codes. Moreover, the equivalence between $\Z_2\Z_2[u]$-duality and $\Z_2$-duality (Corollary \ref{dualigual}), suggests that $\Z_2\Z_2[u]$-linear codes have no meaningful additional properties to those of $\Z_2$-linear codes. However, the partition of the coordinate set into two subsets (the $\Z_2$ and the $\Z_2[u]$ coordinates) open some possible lines of research. In particular, cyclic $\Z_2\Z_2[u]$-linear codes are studied in \cite{Abu,ASA,SrBh}.

Another interesting point is the search for $\Z_2$-linear non-$\Z_2\Z_2[u]$-linear codes. In other words, the search for binary linear codes with automorphism group of odd order, according to Corollary \ref{autgrup}. 

\section*{Acknowledgements}
The author thanks Prof. Josep Rif\`{a} for valuable comments on automorphism groups of linear codes.

%








\end{document}